\documentclass[11pt,twoside]{amsart}
\usepackage{latexsym}
\usepackage{amsmath}
\usepackage{amsfonts}
\usepackage{amssymb}
\usepackage{hyperref}
\usepackage{mathrsfs}

\usepackage{float}
\usepackage{subcaption} 
\usepackage{tikz}
\usetikzlibrary{calc}
\usepackage[mathscr]{euscript}

\usepackage{enumitem} 

\usepackage{xcolor}

\newtheorem{theorem}{Theorem}[section]
\newtheorem{lemma}[theorem]{Lemma}
\newtheorem{corollary}[theorem]{Corollary}
\newtheorem{proposition}[theorem]{Proposition}
\newtheorem{sublemma}{}[theorem]

\numberwithin{equation}{section} %Prefered numbering system

\input{bordersquare.tex}

\newcommand{\cl}{{\rm cl}}

\newcommand{\delete}{\backslash}

\newcommand{\bM}{\mathbb{M}}

\newcommand{\union}{\cup}

\newcommand{\intersect}{\cap}

\usepackage[T1]{fontenc}
\usepackage[utf8]{inputenc}

\title{Circuit-Difference Matroids}
\author[G. Drummond]{George Drummond}
\address{School of Mathematics and Statistics \\
    University of Canterbury \\
    Christchurch, New Zealand}
\email{george.drummond@pg.canterbury.ac.nz}

\author[T. Fife]{Tara Fife}
\address{Mathematics Department\\
	Louisiana State University\\
	Baton Rouge, Louisiana}
\email{fi.tara@gmail.com}

\author[K. Grace]{Kevin Grace}
\address{School of Mathematics \\
    University of Bristol\\
    and
   Heilbronn Institute for Mathematical Research\\ 
    Bristol, United Kingdom}
\email{kevin.grace@bristol.ac.uk}

\author[J. Oxley]{James Oxley}
\address{Mathematics Department\\
	Louisiana State University\\
	Baton Rouge, Louisiana}
\email{oxley@math.lsu.edu}


\begin{document}
\thispagestyle{empty}

\begin{abstract}
One characterization of binary matroids is that the symmetric difference of every pair of intersecting circuits is a disjoint union of circuits. This paper  considers circuit-difference matroids, that is, those  matroids in which the symmetric difference of every pair of intersecting circuits is a single circuit.  Our main result shows that   a connected regular matroid  is circuit-difference if and only if it contains no pair of skew circuits. Using a result of Pfeil, this enables us to explicitly determine all regular circuit-difference matroids. The class of circuit-difference matroids is not closed under  minors, but it is closed under series minors.  We characterize the infinitely many excluded series minors for the class. 
\end{abstract}

\maketitle

\section{Introduction}
A matroid $M$ is \emph{circuit-difference} if $C_1\triangle C_2$ is a circuit whenever $C_1$ and $C_2$ are distinct intersecting circuits of $M$. Evidently, all such matroids are binary. An example of such a matroid is the tipless binary $r$-spike, that is, the matroid  whose binary representation is  $[I_r\mid J_r-I_r]$, where $J_r$ is the $r\times r$ matrix of all ones. 
%$Z_r\delete t$, a geometric representation for which is given in Figure \ref{fig: binary spike}.  Its binary representation is given by $[I_r\mid J_r-I_r]$, where $J_r$ is the $r\times r$ matrix of all ones. 
Subsets $X$ and $Y$ of $E(M)$ are \emph{skew} if $r(X\cup Y) = r(X) + r(Y)$. It is easy to check that no two circuits of the tipless binary $r$-spike are skew. The following is the main result of the paper. 

\begin{theorem}
\label{thm: main}
Let $M$ be a connected regular matroid. Then $M$ is a circuit-difference matroid if and only if it has no pair of skew circuits.
\end{theorem}

To see that this theorem does not extend to all binary matroids, consider the matroid $S_8$ for which a binary representation is shown in 
Figure~\ref{fig: S8}. In this matroid, $\{1,4,7,8\}$ and $\{2,3,5,6,8\}$ are circuits whose symmetric difference is the disjoint union of the circuits $\{1,2,6\}$ and $\{3,4,5,7\}$. Thus $S_8$ is not a circuit-difference matroid. However, since $r(S_8) = 4$ and the only $3$-circuits of $S_8$ contain $6$, the matroid $S_8$ has no two skew circuits. Thus one implication of the last theorem fails for arbitrary connected binary matroids. However, as the next result shows, the other implication  does hold in the more general context.  The proof of this lemma will be given in the next section. 

\begin{figure}[H]
    \centering
\[
\bordersquare{&1&2&3&4&5&6&7&8\cr
&1&0&0&0&1&1&1&0\cr
&0&1&0&0&1&1&1&1\cr
&0&0&1&0&0&0&1&1\cr
&0&0&0&1&1&0&0&1\cr
}.
\]
%    \vspace{0.3cm}
    \caption{A binary representation for $S_8$.}
    \label{fig: S8}
\end{figure}

\begin{lemma}
\label{lem: skew circuits -> not circuit-difference}
Let $M$ be a connected binary matroid. If $M$ has a pair of skew circuits, then $M$ is not circuit-difference. 
\end{lemma}

The terminology used here will follow \cite{James} with one exception. In  \cite{James}, the term ``series extension'' for matroids is defined as the addition of an element $e$ to a matroid $M$
to create a matroid $M'$ in which $\{e,f\}$ is a cocircuit where $f$ is an element of $M$, and $M'/e= M$. It will be expedient here to use the term ``series extension'' more broadly. We shall call a matroid $M'$ a \emph{series extension} of $M$ if it is obtained  from $M$ by a sequence of one-element 
series-extension moves.

Pfeil~\cite{simon} defined a connected matroid $M$ to be \emph{unbreakable} if $M/F$ is connected for every flat $F$ of $M$. He proved that a matroid is unbreakable if and only if its dual has no two skew circuits, and he determined all unbreakable regular matroids. Combining Pfeil's two results gives the following. 

\begin{theorem}
\label{simon_thm} 
A non-empty connected regular matroid $M$ has no two skew circuits if and only if $M$ is a series extension of  one of the following matroids:  $U_{0,1}$, $U_{1,m}$ for some $m \ge 1$; $M^*(K_n)$ for some $n\ge 1$; $M(K_{3,3})$; or $R_{10}$.
\end{theorem}

By combining this theorem with our main result, we have the following explicit description of all regular circuit-difference matroids. 

\begin{corollary}
\label{main_cor}
A  regular matroid $M$ is circuit-difference if and only if every component of $M$  is a series extension of one of the following matroids: $U_{0,1}$, $U_{1,m}$ for some $m \ge 1$; $M^*(K_n)$ for some  $n \ge 1$; $M(K_{3,3})$; or $R_{10}$.
\end{corollary}

One can check explicitly, or deduce from the last corollary, that $M(K_4)$ is a circuit-difference matroid, but  that $M(K_4)/e$ is not circuit-difference for each element $e$. Thus the class of circuit-difference matroids is not minor-closed. We shall show that this class is closed under series minors and will characterize the  excluded series minors in Section \ref{sec: excluded minors}. 
In the next section, we prove some auxiliary results that will be used in the proofs of the main results. The proof of Theorem~\ref{thm: main} will be given in the third section.

\section{Preliminary Results}

%The results in this section will be used in the proof of the main result and in the determination of the excluded series minors for the class of circuit-difference matroids. 

We begin this section with the proof delayed from the last section.

\begin{proof}[Proof of Lemma~\ref{lem: skew circuits -> not circuit-difference}.]
Let $C_1$ and $C_2$ be skew circuits of $M$, and let $D$ be a circuit that meets $C_1$ and $C_2$.  Since $C_1$ and $C_2$ are skew, $|D-(C_1\union C_2)| > 0$. As is easily checked, 
%can be seen from the Venn diagram in Figure \ref{fig:venn diagram}, 
$$D-(C_1\union C_2) = (C_1\triangle D) \intersect (C_2\triangle D).$$
Thus $C_1\triangle D$ meets $C_2\triangle D$. As their symmetric difference is the disjoint union of the circuits $C_1$ and $C_2$, we deduce that 
$M$ is not circuit-difference for either $C_1\triangle D$ or  $C_2\triangle D$ is not a  circuit, or both are circuits but their symmetric difference is not.
\end{proof}

%%%%%%%%%%%%%%%%%%%%%%%%%%%%%%%%%%%%%%%%%%%%%%%%%%%%%%% SECTION BREAK %%%%%%%%%%%%%%%%%%%%%%%%%%%%%%%%%%%%%%%%%%%%%%%%%%%%%%%%%%%%%%%%%%%

The straightforward proof of the next lemma is omitted.

\begin{lemma}\label{lem: A}
Let $M$ be a matroid. If $M$ has a pair of skew circuits, then so does every series extension of $M$.
\end{lemma}

The next result makes repeated use of the fact that if a circuit in a matroid meets a $2$-cocircuit, then it contains that $2$-cocircuit.

\begin{lemma}\label{lem: B}
Let $M$ be a circuit-difference binary matroid and suppose that $M'$ is obtained from $M$ by adding an element $e$ in series to an element $f$ of $M$. Then $M'$ is circuit-difference.
\end{lemma}

\begin{proof}
Let $D_1$ and $D_2$ be an intersecting pair of circuits of $M'$. Suppose first that $e\in D_1\cap D_2$. Then $f\in D_1\cap D_2$ and $\{D_1-e,D_2-e\}$ is an intersecting pair of circuits of $M'/e$. Thus $(D_1-e)\triangle (D_2-e)$, which equals $D_1\triangle D_2$, is a circuit of $M'/e$. Hence $D_1\triangle D_2$ or $(D_1\triangle D_2)\cup e$ is a circuit of $M'$. Because $f\notin D_1\triangle D_2$, the latter cannot occur. Hence $D_1\triangle D_2$ is a circuit of $M'$.

Assume next  that $e\in D_1-D_2$. Then $f\in D_1-D_2$. Now, $D_1-e$ is a circuit of $M'/e$. Moreover, $D_2$ is a circuit of $M'/e$ as otherwise $M'$ would have a circuit that contains $e$ and is contained in $D_2\cup e$. As such a circuit would avoid $f$, we have a contradiction. We now know that $(D_1-e)\triangle D_2$ is a circuit of $M'/e$ containing $f$, so $D_1\triangle D_2$ is a circuit of $M'$.

Finally, assume that $e\notin D_1\cup D_2$. Then $f\notin D_1\cup D_2$ and so $D_1$ and $D_2$ are circuits of $M'/e$. Hence so is $D_1\triangle D_2$. As this set avoids $f$, it must also be a circuit of $M'$ and the lemma is proved.
\end{proof}

In the proof of Theorem \ref{thm: main}, we will encounter a matroid with the property that the complement of every circuit is a circuit. We call such matroids \emph{circuit-complementary}. Such matroids that are binary form an interesting subclass of the class of circuit-difference matroids and are crucial in Section~\ref{sec: excluded minors} when considering the excluded series minors for the latter class.
\begin{lemma}\label{lem: D}
Let $M$ be a connected binary matroid that is circuit-comple\-mentary. Then  $M$ is a circuit-difference matroid.
\end{lemma}
\begin{proof}
%If $M\cong U_{1,1}$, then $M$ is obviously circuit-difference. Otherwise, 
Let $C_1$ and $C_2$ be an intersecting pair of circuits of $M$. Then $C_1\triangle C_2$ is a disjoint union of circuits. If there are at least two circuits in this union, then, since this union avoids $C_1\cap C_2$, we violate the property that the complement of every circuit is a circuit.
\end{proof}

Again, the proof of the next result is elementary and is omitted.

\begin{lemma}\label{lem: E}
Let $M$ be a connected binary matroid that is circuit-comple\-mentary. %Then the following hold.
\begin{enumerate}[label=\textbf{\emph{(\roman*)}},leftmargin=*,align=left]
\item If $\{e,f\}$ is a cocircuit of $M$,  then $M/e$ is circuit-complementary.
\item If $M'$ is a series extension of $M$, then $M'$ is circuit-complementary.  
\end{enumerate}
\end{lemma}

\begin{lemma}\label{lem: F}
Let $M$ be a cosimple connected graphic matroid that is circuit-complementary. Then $M\cong U_{1,4}$.
\end{lemma}
\begin{proof}
Let $M=M(G)$. By Lemmas~\ref{lem: skew circuits -> not circuit-difference} and \ref{lem: D}, $M$ has no two skew circuits. Let $C$ be a cycle of $G$. Then $E(G)-C$ is a cycle $C'$ of $G$. Now, $C$ and $C'$ must have exactly two common vertices, otherwise $G$ is not 2-connected or $M$ has two skew circuits. It follows that $G$ has two vertices $u$ and $v$ that are joined by four internally disjoint paths where these paths use all of the edges of $G$. As $M(G)$ is cosimple, we deduce that $M\cong U_{1,4}$.
\end{proof}

The following lemma makes repeated use of the fact that in a loopless $2$-connected graph, the set of edges meeting a  vertex is a bond.
\begin{lemma}\label{lem: G}
Let $M$ be a cosimple connected cographic matroid that is circuit-complementary. Then $M\cong U_{1,4}$.
\end{lemma}

\begin{proof}
Let $M=M^*(G)$. Then $G$ is $2$-connected and simple. Take a vertex $v$ of $G$ and let $C_1$ be the set of edges meeting $v$. Then $C_1$ is a bond in $G$ and hence a circuit of $M$. Thus $E(G) - C$ is also   a bond of $G$. Hence $G$ has a vertex $w$ that is not adjacent to $v$. Let $C_2$ be the set of edges meeting $w$. Then $E(G)=C_1\cup C_2$ and  $G$ is isomorphic to $K_{2,n}$ for some $n\geq 2$. Let $u$ be a vertex of $G$ other than $v$ or $w$. The complement of the set of edges meeting $u$ is a bond of $G$. Thus $n=2$ and $G$ is a $4$-cycle. Hence $M\cong U_{1,4}$.
\end{proof}

\begin{lemma}\label{lem: M isomorphic to U14 or R_10}
Let $M$ be a connected cosimple regular matroid that is circuit-complementary. Then $M$ is isomorphic to $U_{1,4}$ or $R_{10}$. 
\end{lemma}
\begin{proof}
If $M$ is graphic or cographic, then, by Lemmas \ref{lem: F} and \ref{lem: G}, $M\cong U_{1,4}$. Now assume that $M$ is neither graphic nor cographic and is not isomorphic to $R_{10}$. Then, by Seymour's Regular Matroids Decomposition Theorem \cite{Seymour}, as $M$ is connected, it can be obtained from graphic matroids, cographic matroids and copies of $R_{10}$ by a sequence of $2$-sums and $3$-sums. Moreover, each matroid that is used to build $M$ occurs as a minor of $M$.
\begin{sublemma}\label{M 3-connected}
$M$ is $3$-connected.
\end{sublemma}
If $M$ is not $3$-connected, then $M$ has a $2$-separation, $(X,Y)$. Then $M$ is the $2$-sum, with basepoint $p$ say, of matroids $M_X$ and $M_Y$ with ground sets $X\cup p$ and $Y\cup p$, respectively. As $M$ is connected, so are $M_X$ and $M_Y$. Suppose $X$ is independent in $M$. Then $M_X$ must be a circuit with at least three elements. Thus $M$ is not cosimple, a contradiction. We may now assume that both $X$ and $Y$ contain circuits of $M$. Hence, by the circuit-complementary property, both $X$ and $Y$ are circuits of $M$. As $r(X)+r(Y)=r(M)+1$, we see that $(|X|-1)+(|Y|-1)=r(M)+1$ and, consequently,   $r^*(M)=|X|+|Y|-r(M)=3$. Then $M^*$ is a rank-$3$ simple binary connected matroid having $X$ and $Y$ as disjoint cocircuits. It follows that $M^*$ is graphic, a contradiction. We conclude that \ref{M 3-connected} holds.

We may now assume that there are matroids $M_1$ and $M_2$ each with at least seven elements such that $E(M_1)\cap E(M_2)$ is a triangle $T$ in both matroids and $M$ is the $3$-sum of $M_1$ and $M_2$ across this triangle. Moreover, $M_1$ and $M_2$ are both minors of $M$, and $E(M_i)-T$ spans $T$ in $M_i$ for each $i$. Let $X_i=E(M_i)-T$. Then $(X_1,X_2)$ is a $3$-separation of $M$. Suppose $X_1$ is independent in $M$. As $X_1$ spans $T$, it follows that $M_1$ has rank $|X_1|$, so $M_1^*$ has rank three and has $T$ as a triad. Since $M$ is cosimple, no element of $X_1$ is in a 2-circuit of $M_1^*$. As $M_1$ is binary, it follows that $|X_1|\leq 3$ so $|E(M_1)|\leq 6$, a contradiction. We conclude by the circuit-complementary property that both $X_1$ and $X_2$ must be circuits of $M$. Then, as $r(X_1)+r(X_2)=r(M)+2$, we have $r(M)=|X_1|+|X_2|-4$ so $M^*$ has rank four, has $(X_1,X_2)$ as a $3$-separation and has each of $X_1$ and $X_2$ as a cocircuit. Since $M^*$ is a disjoint union of cocircuits, it is affine. As $M^*$ is simple,   $|E(M^*)| \le 8$. But          
$|X_i|\geq 4$ for each $i$,  so $|E(M^*)| = 8$ and 
$M^*\cong AG(3,2)$. This contradicts the fact that $M$ is regular and thereby completes the proof of the lemma.
\end{proof}

\begin{lemma}\label{lem: I}
Let $M$ be a connected regular matroid and let $X$ be a series class in $M$. Then $M\delete X$ and $M/X$ cannot both be connected and circuit-complementary.
\end{lemma}
\begin{proof}
Assume that the lemma fails. By Lemma \ref{lem: M isomorphic to U14 or R_10}, each of $M\delete X$ and $M/X$ is a series extension of $U_{1,4}$ or $R_{10}$. Note that $X$ is independent in $M$. For every series extension $M'$ of $U_{1,4}$ or $R_{10}$, we have that 
$$(r(M'),r^*(M')) \in \{(k+1,3),(k+5,5): k\geq 0\}.$$
Thus, for some non-negative integer $m$, %we have
$$(r(M\delete X),r^*(M\delete X)) \in \{(m+1,3),(m+5,5)\}.$$
Now let $|X|=t$. Then 
$(r(M),r^*(M))\in \{(m+t,4),(m+4+t,6)\}.$
Thus
$(r(M/X),r^*(M/X))\in \{(m,4),(m+4,6)\},$
so $M/X$ cannot be a series extension of $U_{1,4}$ or $R_{10}$, a contradiction.
%and the lemma follows.
\end{proof}

Although the following lemma is well known,  we include a proof for completeness.
\begin{lemma}\label{lemma I'}
Let $Y$ be a set in a connected matroid $M$ such that $|Y|\geq 2$ and $M|Y$ is connected. Let $W$ be a minimal non-empty subset of $E(M)-Y$ such that $M$ has a circuit $C$ such that $C\cap Y\neq \emptyset$ and $C-Y=W$. Then $W$ is a series class of $M|(Y\cup W)$.  
\end{lemma}
\begin{proof}
This is certainly true if $|W|=1$. Now, suppose that $d_1$ and $d_2$ are distinct elements of $W$ that are not in series in $M|(Y\cup W)$. Then $M|(Y\cup W)$ has a circuit $K$ containing $w_1$ and not $w_2$. As $W$ is independent, $K$ meets $Y$. But $K\cap W\subseteq W-w_2$. Thus we have a contradiction to the choice of $W$. We deduce that every two elements of $W$ are in series in $M|(Y\cup W)$. Since $M|Y$ is connected, no element of $Y$ is in series with an element of $W$. Thus $W$ is indeed a series class of $M|(Y\cup W)$.
\end{proof}

%%%%%%%%%%%%%%%%%%%%%%%%%%%%%%%%%%%%%%%%%%%%%%%%%%%%%%% SECTION BREAK %%%%%%%%%%%%%%%%%%%%%%%%%%%%%%%%%%%%%%%%%%%%%%%%%%%%%%%%%%%%%%%%%%%

\section{The Proof of Main Theorem}

In this section, we prove the main result of the paper. 

\begin{proof}[Proof of Theorem \ref{thm: main}.]
Let $M$ be a regular connected matroid. By Lemma \ref{lem: skew circuits -> not circuit-difference}, if $M$ has a pair of skew circuits, then $M$ is not circuit-difference. To prove the converse, consider all connected regular matroids with no two skew circuits that are not circuit-difference, and choose $M$ to be such a matroid with the minimum number of elements. Then, by Lemma~\ref{lem: B}, $M$ is cosimple. Let $C_1$ and $C_2$ be a pair of intersecting circuits of $M$ such that $C_1\triangle C_2$ is not a circuit and $|C_1\cup C_2|$ is a minimum among such pairs. As $M|(C_1\cup C_2)$ is connected, we must  have that $E(M)=C_1\cup C_2$ by our choice of $M$. Now, $C_1\triangle C_2$ is a disjoint union of at least two circuits. 

\setcounter{theorem}{1}
\begin{sublemma}\label{sublem: J1}
If $D$ is a circuit of $M$ contained in $C_1\triangle C_2$, then 
$(C_1\triangle C_2)-D$ is a circuit of $M$.
\end{sublemma}
Clearly, $D$ meets both $C_1-C_2$ and $C_2-C_1$ but contains neither of these sets. The choice of $\{C_1,C_2\}$ implies that $C_1\triangle D$ is a circuit and hence that $(C_1\triangle D)\triangle C_2$ is a circuit. The last set is $(C_1\triangle C_2)-D$, so \ref{sublem: J1} holds.

Let $Z=C_1\triangle C_2$. As $M$ has no two skew circuits, $M|Z$ is connected and, by \ref{sublem: J1}, it is circuit-complementary. 
Thus, by Lemma \ref{lem: M isomorphic to U14 or R_10}, $M|Z$ is a series extension of $U_{1,4}$ or of $R_{10}$. Let $X$ be a minimal non-empty subset of $C_1\cap C_2$ such that $M$ has a circuit whose intersection with $C_1\cap C_2$ is $X$. Then, by Lemma \ref{lemma I'}, $X$ is a series class of $M|(Z\cup X)$. Thus every circuit of $M|(Z \cup X)$ that meets $X$ must contain $X$. 

\begin{sublemma}\label{sublem: 3.3.1'}
Every circuit of $M|Z$ is a circuit of $(M|(Z\cup X))/X$.
\end{sublemma}

Let $D$ be a circuit of $M$ that is contained in $Z$. Then $D$ meets both $C_1-C_2$ and $C_2-C_1$ and, by \ref{sublem: J1}, $Z-D$ is a circuit of $M$ that also meets both $C_1-C_2$ and $C_2-C_1$. Assume that $D$ is not a circuit of $(M|(Z\cup X))/X$. Then  $M|(Z\cup X)$ has a circuit $K$ such that $K \subseteq D \cup X$ and 
$K\cap D \neq D$. Thus $K$ meets and so contains $X$. Hence $K- D= X$.
As $|K \cup D| = |X \cup D| < |C_1 \cup C_2|$, it follows that $K\triangle D$ is a circuit of $M$ and hence that $K\triangle D$ meets $C_1 - C_2$ and $C_2 - C_1$.

As $|C_1\cup K|< |C_1 \cup D| < |C_1\cup C_2|$, 
  the choice of $\{C_1,C_2\}$ implies that $C_1\triangle K$ is a circuit $C$ of $M$ and that $C_1\triangle (Z-D)$ is a circuit $C'$ of $M$. As $C$ and $C'$ both contain the non-empty set $(D-K)\cap C_1$ and both avoid the non-empty set $(D-K)\cap C_2$, we see that $|C\cup C'|<|C_1\cup C_2|$ and $C\triangle C'$ is a circuit of $M$. This circuit is 
$[C_1\triangle K]\triangle [C_1\triangle (Z-D)],$ which equals $K\triangle (Z-D)$. But the last set is a disjoint union of two circuits, a contradiction. Thus \ref{sublem: 3.3.1'} holds.

We know that $M|Z$ is connected and circuit-complementary. Moreover,   the choice of $X$ implies that $M|(Z \cup X)$  has a circuit that meets $C_1 \cap C_2$ in $X$. Therefore $M|(Z \cup X)$ is connected. Moreover, by \ref{sublem: 3.3.1'}, $(M|(Z \cup X))/X$ is connected. It follows by 
 Lemma~\ref{lem: I} that $(M|(Z\cup X))/X$ is not circuit-complementary. Thus $(M|(Z\cup X))/X$ has a circuit $J$ such that $Z-J$ is not a circuit of 
$(M|(Z\cup X))/X$.  If $J$ is a circuit of $M|Z$, then, as $M|Z$ is circuit-complementary, $Z-J$ is a circuit of $M|Z$. Thus, by \ref{sublem: 3.3.1'}, we obtain the contradiction that $Z-J$ is a circuit of $(M|(Z\cup X))/X$. We deduce that $J$ is not a circuit of $M|Z$. Then $J \cup X'$ is a circuit $K$ of $M|Z$ for some non-empty subset $X'$ of $X$. By the choice of $X$, it follows that $X' = X$. 
Now, $Z=D\cup D'$ for some disjoint circuits $D$ and $D'$. We deduce using  \ref{sublem: 3.3.1'} that  $K$ meets both $D$ and $D'$ but contains neither. Hence $|K\cup D|< |C_1\cup C_2|$, so $K\triangle D$ is a circuit of $M$. As $|D'\cup (K\triangle D)|<|C_1\cup C_2|$, we see that $D'\triangle(K\triangle D)$ is a circuit of $M$, that is, $(Z-K)\cup X$ is a circuit of $M$. Thus $Z-J$ is a circuit of $(M|(Z\cup X))/X$, a contradiction.
\end{proof}

%%%%%%%%%%%%%%%%%%%%%%%%%%%%%%%%%%%%%%%%%%%%%%%%%%%%%%% SECTION BREAK %%%%%%%%%%%%%%%%%%%%%%%%%%%%%%%%%%%%%%%%%%%%%%%%%%%%%%%%%%%%%%%%%%%
\section{Excluded Series Minors of Circuit-Difference Matroids}\label{sec: excluded minors}

In this section, we show  that the class of circuit-difference matroids is closed under series minors, and we characterize the infinitely many excluded series minors for this class.

\begin{lemma}
	The class of circuit-difference matroids is closed under series minors.
\end{lemma}
\begin{proof}
	Let $M$ be a circuit-difference matroid. Evidently, $M\delete e$ is circuit-difference for all $e\in E(M)$. Now let $\{e,f\}$ be a cocircuit of $M$ and consider $M/e$. A circuit $C$ of $M/e$ contains $f$ if and only if $C\cup e$ is a circuit of $M$. Thus the collection ${\mathcal C}(M/e)$ of circuits of $M/e$ is ${\mathcal C}(M\delete e) \cup \{C-e: f \in C \in {\mathcal C}(M)\}.$ It is now routine to check that $M/e$ is a circuit-difference matroid. 
\end{proof}

Let $N_5$ be the $5$-element matroid that is obtained from a triangle by adding   single elements in parallel to exactly  two of its elements. This is easily seen to be an excluded series minor for the class of circuit-difference matroids. Although the next proposition is not needed for the proof of the main result of this section, it seems to be of independent interest.

\begin{proposition}\label{lem: C}
	A connected binary matroid $M$ has a pair of skew circuits if and only if $M$ has a series minor isomorphic to $N_5$.
\end{proposition}

\begin{proof}
	If $M$ has  a series minor isomorphic to $N_5$, then, by Lemma~\ref{lem: A}, as $N_5$ has a pair of skew circuits, so does $M$. For the converse, let $C_1$ and $C_2$ be a pair of skew circuits of $M$, and let  $D$ be a circuit meeting both such that $|D-(C_1\cup C_2)|$ is a minimum. %Then both $C_1\triangle D$ and $C_2\triangle D$ are circuits of $M$. 
	Let $M'=M|(C_1\cup C_2\cup D)$. Next  we show the following.
	\begin{sublemma}\label{C.1}
		If $C_1-D$ or $C_1\cap D$ contains $\{x,y\}$, then $\{x,y\}$ is a cocircuit of $M'$.
	\end{sublemma}
	Suppose that this fails. Then $M'$ has a circuit $K$ that contains $x$ but not $y$. Assume first that $K$ meets $C_2$. Then, by the choice of $D$, we must have that $K-(C_1\cup C_2)=D-(C_1\cup C_2)$. Then $K\triangle D$ is a disjoint union of circuits that is contained in $(C_1\cup C_2)-y$ or $(C_1\cup C_2)-x$. But, for each $z$ in $C_1$, the matroid $(M|(C_1 \cup C_2)) \delete z$ has $C_2$ as its only circuit. As $K\triangle D \neq C_2$, we have a contradiction. 
 We deduce that $K$ avoids $C_2$. As $y\notin K$, we must have that $K\cap (D-(C_1\cup C_2))$ is non-empty. Then $K\triangle D$ is a disjoint union of circuits that does not contain $D-(C_1\cup C_2)$. One such circuit must meet $C_2\cap D$ and $C_1$. But this violates the choice of $D$. Thus \ref{C.1} holds.
	
	By \ref{C.1} and symmetry, we can perform a sequence of series contractions in $M'$, reducing each of the sets $C_1-D$, $C_1\cap D$, $C_2\cap D$, and $C_2-D$ to a single element. The resulting matroid is a series minor of $M$ that has two disjoint $2$-circuits such that deleting one element from each leaves a circuit with at least three elements. It follows that $M$ has $N_5$ as a series minor.
\end{proof}

We call a matroid {\em hyperplane-complementary} if the complement of every hyperplane is a hyperplane. 
One such matroid is the binary affine geometry $AG(r-1,2)$ of rank at least two. The next result determines all simple binary hyperplane-complementary matroids.  For all $k$, every  rank-$k$ flat of $AG(r-1,2)$ is isomorphic to $AG(k-1,2)$.

\begin{lemma}\label{lem: characterizing hyp-comp}
A simple rank-$r$ binary matroid $M$ is hyperplane-complement\-ary if and only if $r \ge 2$ and 
$M \cong AG(r-1,2)\delete X$  for some set  $X$ such that $AG(r-1,2)|X$ does not contain a copy of $AG(r-3,2)$.
\end{lemma}
\begin{proof}
Suppose that $M$ is hyperplane-complementary. Then $r \ge 2$. Moreover, $E(M)$ is a disjoint union of cocircuits, so every circuit of $M$ has even cardinality. Hence 
we can view $M$ as $AG(r-1,2)\delete X$ for some set $X$. Let $E = E(AG(r-1,2))$. Then $E(M) = E - X$. 
Assume that $AG(r-1,2)|X$  contains a copy $Z$ of $AG(r-3,2)$. For $y\in E-X$, consider the closure $\cl_A(Z\cup y)$ of $Z \cup y$ in $AG(r-1,2)$. This closure is a rank-$(r-1)$ flat of $AG(r-1,2)$ and is thus isomorphic to $AG(r-2,2)$. Let $Y=\cl_A(Z\cup y)\cap (E-X)$ and $W=(E-X)-Y$.  Then $Y$ is contained in some copy of $AG(r-3,2)$, and $W$ is contained in some copy of $AG(r-2,2)$. Thus $r(Y) \le r-2$ and $r(W) \le r-1$. Hence $W$ is contained in a hyperplane $W'$ of $M$ whose complement  in $E(M)$ is not a hyperplane. Thus $M$ is not hyperplane-complementary, a contradiction.

Now let $M = AG(r-1,2)\delete X$  where $r \ge 2$ and $AG(r-1,2)|X$  does not contain a copy of $AG(r-3,2)$. Let $H$ be a hyperplane of $AG(r-1,2)$. 
Then $AG(r-1,2)|H = AG(r-2,2)$. 
If $r(H-X)\leq r-2$, then $H-X$ is contained in some copy of $AG(r-3,2)$  that is contained in $H$ and so, as $AG(r-2,2)$ is hyperplane-complementary, $X$ contains a copy of $AG(r-3,2)$. This contradiction implies that the hyperplanes of $M$ are all of the   sets of the form $H-X$ where $H$ is a hyperplane of $AG(r-1,2)$. As  $AG(r-1,2)$ is  hyperplane-complementary, so is  $M$.
\end{proof}

Recall that $AG(r-1,2)$  is obtained from the projective geometry $PG(r-1,2)$ by deleting a hyperplane, that is, by deleting  a copy of $PG(r-2,2)$.  It is a well-known consequence of the unique representability of binary matroids that if $PG(r-1,2)|E_1\cong PG(r-1,2)|E_2$, then $PG(r-1,2)\delete E_1\cong PG(r-1,2)\delete E_2$. Thus, as all single-element deletions of $PG(r-2,2)$ are isomorphic,   there is, up to isomorphism,    a unique simple binary rank-$r$ single-element extension of  $AG(r-1,2)$. We shall denote this extension by  $AG(r-1,2)+ e$.

Let $\bM$ be the set of all matroids of rank at least three of the form $[AG(r-1,2)+e]\delete X$ such that $AG(r-1,2)\delete X$ is hyperplane-complementary of rank $r$. Thus $N_5^*$ is the unique rank-$3$ member of $\bM$ while its rank-$4$ members are the tipped binary $4$-spike and a non-tip deletion thereof, that is,  $S_8$. We now  show that the duals of the matroids in $\bM$ are precisely the  excluded series minors for the class of circuit-difference matroids.

\begin{theorem}
	A binary matroid $M$ is an excluded series minor for the class of circuit-difference matroids if and only if  $M^*\in \bM$.
\end{theorem}
\begin{proof}
	Let $M$ be an excluded series minor for the class of circuit-difference matroids. By Lemma~\ref{lem: B}, $M$ is cosimple. 
	Let $C_1$ and $C_2$ be intersecting circuits of $M$  such that $C_1 \triangle C_2$ is not a circuit and $|C_1\cup C_2|$ is minimal. 
	\begin{sublemma}\label{4.7.1}
		$M^*\in \bM$.
	\end{sublemma}
	Evidently, $E(M)=C_1\cup C_2$. Then $C_1$ and $C_2$ are the only circuits of $M$ containing $C_1 - C_2$ and $C_2 -  C_1$, respectively. 
	Now, letting $x\in C_1\cap C_2$, suppose that $(C_1 \cap C_2) - x$ contains an element $y$. 
Then, as $x$ and $y$ are not in series,  $M$ has a circuit $D$ containing $x$ but not $y$. As $D$   meets both $C_1$ and $C_2$, we have, by the choice of $\{C_1,C_2\}$, that $C_1\triangle D$, $C_2\triangle D$, and hence $(C_1\triangle D)\triangle (C_2\triangle D)$ are circuits of $M$. This last circuit is $C_1 \triangle C_2$, so we have a contradiction. Thus $C_1\cap C_2 = \{x\}$. 
	
To see that $M/x$ is circuit-complementary, let $D$ be a circuit of $M/x$ other than $C_1 - x$ or $C_2 - x$. Then either $D$ or $D\cup x$ is a circuit $D'$ of $M$, and  $D'$ must meet $C_1-C_2$ and $C_2 - C_1$.  %Letting $D^c=C_1\triangle C_2-D$, we then have that either $D^c$ or $D^c\cup x$ is a circuit of $M$. 
We show next that  $(C_1 \triangle C_2)  - D$ is a circuit of $M/x$. Suppose it is not. Then $x \not\in D'$ and $M$ has a circuit $D''$ containing $x$ such 
that $D'' \subsetneqq(C_1 \triangle C_2) - D$. By the choice of $\{C_1,C_2\}$, we see that $C_1\triangle D''$ and hence $(C_1\triangle D'')\triangle C_2$ is a circuit of $M$.
As this circuit properly contains   $D''$, we have a contradiction. We conclude that  $(C_1 \triangle C_2)  - D$ is a circuit of $M/x$, so 
$M/x$ is circuit-complementary. 
Therefore, $M^*\delete x$ is hyperplane-complementary. As $M$ is cosimple, $M^*$ is simple. Moreover, $M^*$ has $C_1 - x$ and $C_2 - x$ as hyperplanes, so $M^*$ has the form $[AG(r-1,2) + e]\delete X$ where $AG(r-1,2) \delete X$ is hyperplane-complementary of rank $r$. 
Since $M^*$ is connected, $r(M^*) \ge 2$. But if $r(M^*) = 2$, then $M^* \cong U_{2,3}$, so $M \cong U_{1,3}$ and $M$ is circuit-difference, a contradiction. 
Thus  $M^*\in \bM$, so \ref{4.7.1} holds.
	
	To prove the converse, let  $M^* = [AG(r-1,2) + e] \delete X$ where $AG(r-1,2)\delete X$ is hyperplane-complementary of rank $r$ and $r \ge 3$. 
By Lemma~\ref{lem: characterizing hyp-comp}, $AG(r-1,2)|X$	does not contain a copy of $AG(r-3,2)$. Consider $AG(r-1,2) + e$ and let $H_0$ be the hyperplane of $PG(r-1,2)$ whose deletion gives $AG(r-1,2)$. Take a rank-$(r-2)$ flat $F$ of $PG(r-1,2)$ that is contained in $H_0$ and avoids $e$. Apart from $H_0$, there are exactly two   hyperplanes, $H_1$ and $H_2$, of $PG(r-1,2)$ that contain $F$. Then $H_1 - H_0$ and $H_2 - H_0$ are hyperplanes of $AG(r-1,2) + e$, and $H_1 - (H_0\cup X)$ and $H_2 - (H_0\cup X)$ are hyperplanes of $[AG(r-1,2) + e]\delete X$. The complements of these two hyperplanes  are circuits $C_1,C_2$ of $M$ that meet in the  element $e$.   We now note that $C_1\triangle C_2$ is not a circuit of $M$   otherwise $\{e\}$ is a hyperplane of $M^*$, and we obtain the contradiction that  $r(M^*) \le 2$. Hence $M$ is not circuit-difference.
	
	\begin{sublemma}\label{4.7.2}
		If $D$ is a circuit of $M\delete e$, then $e\notin \cl(D)$.
	\end{sublemma}
	Suppose   that $e\in \cl(D)$ for some circuit $D$ of $M\delete e$. Then there is a partition $\{X_1,X_2\}$ of $D$ such that $X_i$ is a circuit of $M/e$ for both $i$. As $M/e$ is circuit-complementary having $X_1$ and $X_2$ as disjoint circuits, $X_1 \cup X_2 = E(M/e) = E(M)  - e$. This is a contradiction as $X_1 \cup X_2 = D \subsetneqq E(M)  - e$.  Hence \ref{4.7.2} holds.
	
	\begin{sublemma}\label{4.7.3}
		$M\delete f$ is circuit-difference for all $f$ in $E(M)$.
	\end{sublemma}
	%As $M/x$ is circuit-complementary, it follows from \ref{4.7.2}, that so is $M\delete x$. Thus, we may assume that $y\neq x$. Now, if 
	Suppose some $M\delete f$ is not circuit-difference. Then it has a pair of intersecting circuits $D_1,D_2$ such that $D_1 \triangle D_2$ contains a pair of disjoint circuits $K_1,K_2$. Suppose first that both $D_1$ and $D_2$ avoid $e$. Then so do $K_1$ and $K_2$. Thus, by \ref{4.7.2}, none of $D_1, D_2, K_1$, or $K_2$ has $e$ is in its closure. Hence all of $D_1, D_2, K_1$, and $K_2$ are circuits of $M/e$, so $M/e$ is not circuit-difference, a  contradiction. Hence at least one of $D_1$ and $D_2$ must contain $e$, so $f \neq e$. 
	
	Now suppose $e\in D_1-  D_2$ and $e \in K_1$. Then $D_1-e$ and $D_2$ are intersecting circuits of $M/e$ with circuits $K_1-e$ and $K_2$ in their symmetric difference. This   again  contradicts the fact that $M/e$ is circuit-difference. Hence, by symmetry, we must  have that $e\in D_1\cap D_2$.  Consequently, $K_1$, $K_2$, $D_1-e$, and $D_2-e$ are circuits of $M/e$. If $D_1\cap D_2=\{e\}$, then $D_1-e$ and $D_2-e$ are disjoint. Thus their union is $E(M/e)$. But this union avoids $f$, a contradiction. Hence  $(D_1\cap D_2)-e$ must be non-empty. But then $D_1-e$ and $D_2-e$ are intersecting circuits of $M/e$ so their symmetric difference, which equals $D_1 \triangle D_2$, is a circuit of $M/e$. However, this symmetric difference contains $K_1$ and $K_2$, which  are circuits of $M/e$. This contradiction completes the proof of \ref{4.7.3}. 
	
	As $M^*$ is  simple,  $M$ is cosimple and hence no series contractions can be performed. Thus, by \ref{4.7.3}, every series minor of $M$ is circuit-difference and the theorem holds. 
\end{proof}

The next result follows immediately by combining the last theorem with Tutte's excluded-minor characterization of binary matroids~\cite{Tutte}.

\begin{corollary}
A matroid $M$ is an excluded series minor for the class of circuit-difference matroids if and only if $M \cong U_{n,n+2}$ for some $n \ge 2$, or $M^*$ can be obtained from $AG(r-1,2) + e$ for some $r \ge 3$ by deleting some set $X$ such that $e \not\in X$ and $AG(r-1,2)|X$ does not contain a copy of $AG(r-3,2)$. 
\end{corollary}

%%%%%%%%%%%%%%%%%%%%%%%%%%%%%%%%%%%%%%%%%%%%%%%%%%%%%%% SECTION BREAK %%%%%%%%%%%%%%%%%%%%%%%%%%%%%%%%%%%%%%%%%%%%%%%%%%%%%%%%%%%%%%%%%%%


\begin{thebibliography}{99}

%\bibitem{Bollobas} Bollob{\'a}s, B.,    {\it Extremal Graph Theory}, Academic Press, 1978.
\bibitem{James} Oxley, J.,   {\it Matroid Theory}, Second edition, Oxford University Press, New York, 2011. 
\bibitem{simon} Pfeil, S., {\it On Properties of Matroid Connectivity}, Ph.D. dissertation, Louisiana State University, 2016. 
\bibitem{Seymour} Seymour, P.D.,    Decomposition of regular matroids,  \emph{J. Combin. Theory Ser. B} \textbf{28} (1980), 305-359.
\bibitem{Tutte} Tutte, W.,    A homotopy theorem for matroids, I, II,  \emph{Trans. Amer. Math. Soc.} \textbf{88} (1958), 144-174.
 

\end{thebibliography}
\end{document}